\documentclass[10pt,a4paper]{article}
\usepackage[active]{srcltx}

\usepackage[final]{optional}
\usepackage[british,american]{babel}

\usepackage{amsfonts, amsmath, wasysym}
\usepackage{amssymb,amsthm,
paralist
}
\usepackage{
latexsym,
nicefrac,
}
\usepackage[usenames]{color}

\usepackage{url}

\definecolor{darkgreen}{rgb}{0,0.5,0}
\definecolor{darkred}{rgb}{0.7,0,0}
\usepackage[colorlinks, 
citecolor=darkgreen, linkcolor=darkred
]{hyperref}
\usepackage{esint}
\usepackage{bibgerm}
\usepackage[normalem]{ulem}

\theoremstyle{plain}
\newtheorem{lemma}{Lemma}[section]
\newtheorem{thm}[lemma]{Theorem}
\newtheorem{prop}[lemma]{Proposition}
\newtheorem{cor}[lemma]{Corollary}

\theoremstyle{definition}

\newtheorem{conj}[lemma]{Conjecture}

\newtheorem{rmk}[lemma]{Remark}
\numberwithin{equation}{section}
\newcommand{\m}{\ensuremath{{\cal M}}}

\newcommand{\Om}{\Omega}

\newcommand{\R}{\ensuremath{{\mathbb R}}}

\newcommand{\C}{\ensuremath{{\mathbb C}}}

\DeclareMathOperator{\Vol}{vol}

\newcommand{\beq}{\begin{equation}}
\newcommand{\eeq}{\end{equation}}
\newcommand{\beqa}{\begin{equation}\begin{aligned}}
\newcommand{\eeqa}{\end{aligned}\end{equation}}
\newcommand{\brmk}{\begin{rmk}}
\newcommand{\ermk}{\end{rmk}}
\newcommand{\partref}[1]{\hbox{(\csname @roman\endcsname{\ref{#1}})}}

\newcommand{\Rm}{{\mathrm{Rm}}}
\newcommand{\Ric}{{\mathrm{Ric}}}

\newcommand*{\ee}{\mathop{\mathrm{e}}\nolimits}

\newcommand*\dist{\mathop{\mathrm{dist}}\nolimits}

\newcommand*\Isom{\mathop{\mathrm{Isom}}\nolimits}

\newcommand{\Div}{\operatorname{div}}

\newcommand*\Mf{\mathcal{M}}
\newcommand*\Disc{\mathcal{D}}

\newcommand*{\TD}[1][]{\mathsf{T}^{#1\!}\mathcal{D}}

\newcommand*\gBall{\mathcal{B}}

\newcommand*\pddt{\frac{\partial}{\partial t}}

\newcommand*\pdds{\frac{\partial}{\partial s}}

\newcommand*\dx{\,\mathrm{d}x}
\newcommand*\dy{\,\mathrm{d}y}

\newcommand*\dz{\mathrm{d}z}
\newcommand*\Ball{\mathbb{B}}

\title{{\sc
Existence of Ricci flows of incomplete surfaces
}\\ 
}
\author{Gregor Giesen and Peter M. Topping}
\date{\today}

\begin{document}
\maketitle

\begin{abstract}
We prove a general existence result for instantaneously complete
Ricci flows starting at an arbitrary Riemannian surface which
may be incomplete and may have unbounded curvature.
We give an explicit formula for the maximal existence time,
and describe the asymptotic behaviour in most cases.
\end{abstract}

\section{Introduction}
Hamilton's Ricci flow \cite{Ham82} takes a Riemannian metric $g_0$ on
a manifold $\m$ and deforms it under the evolution equation
\begin{equation}
  \label{eq:ricci-flow}
\left\{
\begin{aligned}
  \pddt g(t) &= -2\Ric\bigl[g(t)\bigr]\\
  g(0) &= g_0
\end{aligned}
\right.
\end{equation}
There is now a good well-posedness theory for this PDE which
we summarise in the following theorem.
\begin{thm}{\rm (Hamilton \cite{Ham82}, DeTurck \cite{DeT03}, Shi
    \cite{Shi89}, Chen-Zhu \cite{CZ06}.)}
  \label{shihamilton}
  Given a complete Riemannian manifold $\bigl(\m^n,g_0\bigr)$ with bounded
  curvature $\bigl|\Rm[g_0]\bigr|\le K_0$, there exists $T>0$ depending only
  on $n$ and $K_0$, and a  
  Ricci flow $g(t)$ for $t\in [0,T)$ satisfying 
  \eqref{eq:ricci-flow} with bounded curvature and for which
  $\bigl(\m,g(t)\bigr)$ is complete for each $t\in [0,T)$. Moreover, any
  other complete, bounded curvature Ricci flow $\tilde g(t)$
  with $\tilde g(0)=g_0$
  must agree with $g(t)$ while both flows exist. 
\end{thm}
We will call the Ricci flow whose existence is asserted by
this theorem the \emph{Hamilton-Shi} Ricci flow.

This paper is dedicated to the more general problem
of posing Ricci flow in the case that the initial manifold
$(\m,g_0)$ may be incomplete, and may have curvature unbounded
above and/or below. Without further conditions, this problem
is ill-posed, with extreme nonuniqueness. However, in \cite{Top10}
the second author introduced the restricted class of 
\emph{instantaneously complete} Ricci flows, proving the following
theorem in the case that $\dim\m=2$.

\begin{thm}{\rm (Topping \cite{Top10}.)}
  \label{ICexistence1}
  Let $(\m^2,g_0)$ be any smooth metric on any surface $\m^2$
  (without boundary) with Gaussian curvature bounded above by $K_0\in \R$.
  Then for $T>0$ sufficiently small so that 
  $K_0<\frac{1}{2T}$,
  there exists a smooth Ricci flow $g(t)$ on $\m^2$ for 
  $t\in [0,T]$, 
  such that $g(0)=g_0$ and $g(t)$ is complete
  for all $t\in (0,T]$.
  
  The curvature of $g(t)$ is uniformly bounded above, and
  $g(t)$ is \textbf{maximally stretched}
  in the sense that if $\bigl(\tilde g(t)\bigr)_{t\in[0,\tilde T]}$ 
  is any Ricci flow on $\Mf^2$ with $\tilde g(0)\leq g(0)$ (with $\tilde g(t)$ 
  not necessarily complete or of bounded curvature) then
  \[ \tilde g(t) \le g(t) \quad\text{for every } t\in\bigl[0,\min\{T,\tilde T\}\bigr]. \]
\end{thm}
From now on we will call a Ricci flow $\bigl(\m,g(t)\bigr)_{t\in I}$ 
where $I=[0,T)$ or $I=[0,T]$ \textbf{instantaneously complete}, 
if $g(t)$ is complete for all $t\in I$ with $t>0$.

The significance of this class of instantaneously complete
Ricci flows lies in the conjecture, also from 
\cite{Top10}, that the solution $g(t)$ should be \emph{unique}
within this class. See
Conjecture \ref{uniqueness_conj} below.

In this paper we improve Theorem \ref{ICexistence1}
in three ways. First, we drop the need for the upper curvature 
bound on the initial metric; second, we give a precise formula for the 
maximal existence time in all cases; third we show that
the (rescaled) Ricci flow converges to a hyperbolic metric
whenever there exists such a metric to which it can converge.

\begin{thm}{\rm(Main theorem.)}
  \label{thm:2d-existence}
  Let $\bigl(\Mf^2,g_0\bigr)$ be a smooth Riemannian surface which need not be
  complete, and could have unbounded curvature. 
  Depending on the conformal type, we define $T\in(0,\infty]$ by
  \[ T := \begin{cases}
    \frac1{8\pi}\Vol_{g_0}\Mf & \text{if }(\Mf,g_0)\cong\mathcal S^2,\\
    \frac1{4\pi}\Vol_{g_0}\Mf & \text{if }(\Mf,g_0)\cong\mathbb
    C\text{ or
    }(\Mf,g_0)\cong\mathbb R\!P^2, \\
    \qquad\infty & \text{otherwise}.\footnotemark
  \end{cases} \]
  \footnotetext{Note that also $T=\infty$ if $\Vol_{g_0}\mathbb C=\infty$.}
  Then there exists a smooth Ricci flow $\bigl(g(t)\bigr)_{t\in[0,T)}$
  such that 
  \begin{compactenum}
  \item $g(0)=g_0$;
  \item $g(t)$ is instantaneously complete;
  \item $g(t)$ is maximally stretched,
  \end{compactenum}
  and this flow is unique in the sense that if
  $\bigl(\tilde g(t)\bigr)_{t\in[0,\tilde T)}$ is 
  any other Ricci flow on $\m$ satisfying 1,2 and 3, then
  $\tilde T\leq T$ and $\tilde g(t)=g(t)$ for all $t\in[0,\tilde T)$.

  If $T<\infty$, then we have 
  \[ \Vol_{g(t)}\Mf = \left\{\begin{array}{ll}
      8\pi (T-t) & \text{if }(\Mf,g_0)\cong\mathcal S^2,\\
      4\pi (T-t) & \text{otherwise},
    \end{array}\right\}\quad\longrightarrow\quad 0 \quad\text{ as } t\nearrow
  T, \] 
  and in particular, $T$ is the maximal existence time.
  Alternatively, if $\Mf$ supports a complete
  hyperbolic\footnote{We call a metric $H$ \emph{hyperbolic} if it has
    constant Gaussian curvature $K[H]\equiv -1$.} metric $H$ conformally 
  equivalent to $g_0$ (in which case $T=\infty$) then we have convergence of
  the rescaled solution 
  \[ \frac1{2t}g(t) \longrightarrow H  \quad\text{smoothly locally as }
  t\to\infty. \] 
  If additionally there exists a constant $M>0$ such that $g_0\le MH$ then the
  convergence is global: For any $k\in\mathbb N_0:=\mathbb N\cup\{0\}$ and $\eta\in(0,1)$ there
  exists a constant $C=C(k,\eta,M)>0$ such that for all $t\geq 1$ 
  \[ \left\| \frac1{2t} g(t) - H \right\|_{C^k(\Mf,H)} \le \frac C{t^{1-\eta}}
  \quad\stackrel{t\to\infty}\longrightarrow\quad 0. \]
  In fact, in this latter case, for all $t>0$ we have 
  $\left\| \frac1{2t} g(t) - H \right\|_{C^0(\Mf,H)} 
  \le \frac C{t}$,
  and even
  \[ 0\le \frac1{2t}g(t) -H \le \frac M{2t}H. \]
\end{thm}

Note, for any smooth Riemannian manifold $(\Mf,g)$ we denote the $C^k$ norm
of a smooth tensor field $T$ on $\m$ by
\[ \Bigl\|\,T\,\Bigr\|_{C^k(\Mf,g)} = \sum_{j=0}^k\; \sup_\Mf\, \Bigl|\,
\nabla_{\!g}^j T\, \Bigr|_g. \]

The final part of Theorem \ref{thm:2d-existence} suggests
that the Ricci flow has a uniformising effect in great generality.
Note that in the special case of \emph{compact} surfaces, an
elegant and complete theory\footnote{For a survey see \cite[\S5]{CK04} or
  \cite{Gie07}.} has been developed by Hamilton and Chow to this effect:
\begin{thm}{\rm (Hamilton \cite{Ham88}, Chow \cite{Cho91}.)}
  \label{thm:exist-unique-closed-surface}
  Let $\bigl(\Mf^2,g_0\bigr)$ be a smooth, compact Riemannian surface without
  boundary.
  Then there exists a unique Ricci flow $g(t)$ with $g(0)=g_0$ for all
  $t\in[0,T)$ up to a maximal time 
  \[ T = \begin{cases}
    \frac1{4\pi\chi(\Mf)}\Vol_{g_0} \Mf & \text{if }\chi(\Mf)>0\\
    \quad\infty & \text{otherwise,}
  \end{cases} \]
  where $\chi(\Mf)$ denotes the Euler characteristic of $\Mf$. 
  Moreover, the rescaled solution 
  \[  \begin{cases}
    \frac1{2(T-t)}\, g(t) & \text{if } \chi(\Mf)>0 \\
    \hfill g(t) & \text{if } \chi(\Mf)=0 \\
    \hfill\frac1{2t}\, g(t) & \text{if } \chi(\Mf)<0
  \end{cases} \]
  converges smoothly 
  to a conformal metric of constant
  Gaussian curvature $1,0,-1$ resp. as $t\to T$.
\end{thm}
Although our result proves that the instantaneously complete 
Ricci flow will 
always uniformise a manifold of hyperbolic type, irrespective
of whether it is complete or not, and 
Theorem \ref{thm:exist-unique-closed-surface} 
proves that it does the same in the spherical 
case, the asymptotic behaviour is more involved on the remaining
manifolds whose universal cover is conformally $\C$.
Nevertheless, the techniques of this paper can be applied to that
case, for example to address the conjecture in \cite[\S 1]{IJ09}.

Theorem \ref{thm:2d-existence} asserts not just the existence 
but also the uniqueness of the given Ricci flow, although
only within the class of maximally stretched solutions.
In fact, the message of \cite{Top10} is that uniqueness
should hold in the much more general class in which 
this maximally stretched condition is dropped:

\begin{conj} (Topping \cite{Top10}.)
  \label{uniqueness_conj}
  The solution of Theorem \ref{thm:2d-existence} is \emph{unique}
  within the class of instantaneously complete Ricci flows 
  $g(t)$ with $g(0)=g_0$.
\end{conj}
This conjecture was partially resolved in \cite{GT10}, and
in this paper we make further progress. Precisely, we have
the following two partial results:
\begin{thm}
\label{not_hyp_unique}
In the setting of Theorem \ref{thm:2d-existence}, if
$(\Mf,g_0)$ is \emph{not} conformally equivalent to
any hyperbolic surface and 
$\bigl( \tilde g(t)\bigr)_{t\in[0,\tilde T)}$ is any 
instantaneously complete Ricci flow
with $\tilde g(0)=g_0$, then $\tilde T\leq T$ and 
$\tilde g(t)=g(t)$ for all $t\in [0,\tilde T)$.
\end{thm}
We stress that no curvature assumption is made on the 
competitor $\tilde g(t)$ in the theorem above.
The theorem can be compared to the results of Chen which would imply
this type of strong uniqueness 
in the case that $(\m,g_0)$ is complete, of bounded curvature
and with controlled geometry \cite{Che09}.

The following can be viewed as a generalisation of a result
from \cite{GT10}.
\begin{thm}
  \label{unfinished}
  Let $\bigl(\Mf^2,H\bigr)$ be a complete hyperbolic surface.
  Suppose $\bigl(g_1(t)\bigr)_{t\in[0,T_1]}$ and
  $\bigl(g_2(t)\bigr)_{t\in[0,T_2]}$ are two instantaneously complete Ricci
  flows on $\Mf$ which are conformally equivalent to $H$, with
  \begin{compactenum}[(i)]
  \item $g_1(0)=g_2(0)$;
  \item there exists $M>0$ such that $g_i(0)\le MH$;
  \item there exists $\varepsilon\in\bigl(0,\min\{T_1,T_2\}\bigr]$ such that the
    curvature of each $g_i(t)$ 
    is bounded above for a short time interval $[0,\varepsilon]$.
  \end{compactenum}
  Then $g_1(t)=g_2(t)$ for all $t\in\bigl[0,\min\{T_1,T_2\}\bigr]$. 
\end{thm}

Since the complete uniqueness conjecture has not been 
fully resolved above in the case that the initial metric is
conformally equivalent to some complete hyperbolic metric, 
it is apriori
conceivable that in the case that $(\m,g_0)$ is complete
and of bounded curvature, the Ricci flow we construct in
Theorem \ref{thm:2d-existence} could be different from the
standard Hamilton-Shi solution of Theorem \ref{shihamilton}. 
We rule out this possibility in the following theorem.
\begin{thm}
\label{thm:no-weird-solution}
Let $\bigl(\Mf^2,g_0\bigr)$ be a complete Riemannian surface with 
bounded curvature, and let $\bigl(g(t)\bigr)_{t\in[0,T)}$ be the 
corresponding solution constructed in Theorem \ref{thm:2d-existence}.
Then $g(t)$ agrees with the Hamilton-Shi Ricci flow of 
Theorem \ref{shihamilton} as long as the latter flow exists.
\end{thm}

The combination of Theorem \ref{thm:2d-existence} and Theorem
\ref{thm:no-weird-solution} generalises a number of other
results which have appeared recently, proved with different
techniques:
\begin{itemize}
\item In the special case that the initial surface $\bigl(\Mf,g_0\bigr)$ is
  complete, topologically finite, with negative Euler characteristic
  $\chi(\Mf)<0$, and on each end of $\Mf$ the initial metric $g_0$ is
  asymptotic to a multiple of a hyperbolic cusp metric \cite{JMS09} or of a
  funnel metric \cite{AAR09}, Ji-Mazzeo-Sesum or Albin-Aldana-Rochon
  resp.~show existence and smooth uniform convergence
  of the normalised Ricci flow to the unique complete metric of constant
  curvature in the conformal class of $g_0$. Since the number of cusp or
  funnel ends is finite and their complement is compact we observe that there
  exists a conformally equivalent metric $H$ of constant negative curvature
  with $H\ge g_0$ and we may apply alternatively Theorem \ref{thm:2d-existence}.

\item If $\bigl(\Mf^2,g_0\bigr)$ is a complete Riemannian surface with
  asymptotically conical ends and negative Euler characteristic $\chi(\Mf)<0$,
  then Isenberg-Mazzeo-Sesum show in \cite{IMS10} the existence of a Ricci
  flow $g(t)$ on $\m$ for all $t\in [0,\infty)$, with 
  $g(0)=g_0$, and smooth local convergence
  of the rescaled flow $t^{-1}g(t)$ to a complete metric of
  constant negative curvature and finite area in the conformal class of
  $g_0$. This is a special case of
  Theorem \ref{thm:2d-existence}.
  
\item For an initial metric $g_0$ on the disc $\Disc$ which is bounded above and below by positive multiples of the
  complete hyperbolic metric $H$, Schn\"urer-Schulze-Simon show 
  existence and
  smooth uniform convergence of the normalised Ricci
  flow to $H$ in \cite{SSS10}. Theorem \ref{thm:2d-existence}
  implies that the metrical equivalence can be weakened to $g_0 \le
  MH$ for some $M>0$. Moreover, without this condition we still have
  smooth local  convergence. 
\end{itemize}

The article is organised as follows: After summarising some of the specific
properties of the Ricci flow in two dimensions in the following paragraph,
we show in Section \ref{sect:hyp-upper-barrier} several apriori estimates for
instantaneously complete Ricci flows which have initially a multiple of a
hyperbolic metric as an upper barrier. These include sharp barriers (above and
below) at later times, the smooth convergence of the rescaled flow,
curvature estimates and thus long time existence. The main
ingredients here come from Chen's very general apriori estimate for the scalar 
curvature of a complete Ricci flow \cite{Che09} and Yau's version of the
Schwarz Lemma \cite{Yau73}. In Section
\ref{subsect:ex-disc} we exploit these properties to improve the existence
result from \cite{Top10} (Theorem \ref{ICexistence1}) on the disc to the case
where we might have initially unbounded curvature, by applying the very same
Theorem \ref{ICexistence1} locally where we have both bounded curvature and the
estimates from Section \ref{sect:hyp-upper-barrier}. 
We also prove Theorem \ref{unfinished}.
In Section
\ref{subsect:ex-uniq-C} we prove a lower barrier for Ricci flows on $\mathbb
C$ only requiring the instantaneous completeness of the flow. This barrier is
sufficient to use a comparison principle by Rodriguez-Vazquez-Esteban
\cite{RVE97} and gain uniqueness in this class, leading
to a proof of Theorem \ref{not_hyp_unique}. Finally, in Section
\ref{last_sect}, we bring 
all these ingredients together to prove the Main Theorem 
\ref{thm:2d-existence}.

\subsection*{Ricci flows on surfaces}

Since our results address the two-dimensional case, we briefly
recall some special features of Ricci flows on surfaces.
On a two-dimensional manifold, the Ricci curvature is
simply the Gaussian curvature $K$ times the metric:
$\Ric[g]=K[g]\, g$. The Ricci flow then moves within a fixed conformal
class, and if we pick a local isothermal complex coordinate $z=x+\mathrm{i}y$
and write the metric in terms of a scalar conformal factor $u\in C^\infty(\Mf)$
\[ g=\ee^{2u}|\dz|^2 \]
where $|\dz|^2=\dx^2+\dy^2$, then the evolution of the metric's conformal
factor $u$ under Ricci flow is governed by the nonlinear scalar PDE
\begin{equation}
  \label{eq:ricci-flow-cf}
  \pddt u = \ee^{-2u}\Delta u = -K[u].
\end{equation}
where $\Delta := \frac{\partial^2}{\partial x^2} + \frac{\partial^2}{\partial
  y^2}$ is defined in terms of the local coordinates and we abuse notation
by abbreviating $K[g]$ by $K[u]$.

The definition of a \emph{maximally stretched} Ricci flow can now be viewed 
as a height-maximality of the conformal factor, i.e. $\ee^{2u(t)}|\dz|^2$ is
maximally stretched if and only if we have $u(t)\ge v(t)$ for any other
conformal solution $\ee^{2v(t)}|\dz|^2$ with $u(0)\ge v(0)$. 

There is an independent interest and extensive
literature on \eqref{eq:ricci-flow-cf} which after the change
of variables $v=\ee^{2u}$ is called the \emph{logarithmic fast diffusion equation}:
\begin{equation}
  \label{eq:log-fde}
  \pddt v = \Delta\log v.
\end{equation}
In a physical context it models the evolution of the thickness of a thin
colloidal film spread over a flat surface if the van der Waals forces are
repulsive. For details we refer to \cite{DdP95}, \cite{DD96} and \cite{RVE97}
plus the references therein. 
Virtually all of the literature in this direction considers the equation on $\R^2$,
and we will appeal to some of these results in the case
that the universal cover of the surface is conformally $\mathbb C$.
Hui (e.g. \cite{Hui02}) has considered the problem on bounded domains in $\R^2$.
Typically the logarithmic fast diffusion literature considers solutions with some 
sort of growth condition at infinity, sometimes phrased in terms of membership 
of an $L^p$ space, and in this paper we
are replacing these conditions with the instantaneously complete condition.

\medskip

\emph{Acknowledgements:} Both authors were supported
by The Leverhulme Trust.

\section{Ricci flows with an upper hyperbolic barrier}
\label{sect:hyp-upper-barrier}

In this section we will derive some estimates for Ricci flows
mainly under the assumption that they 
are initially bounded from above by some (possibly large)
multiple of a hyperbolic metric. No curvature assumptions will
be made on the flow at any time, and yet we derive pointwise
and also derivative bounds which will be of fundamental
importance in the later proofs.
In particular, we apply these estimates even when there is no
hyperbolic metric conformally equivalent to the initial metric
of the Ricci flow; the trick will be to restrict to smaller
compactly contained subdomains where such a metric will exist.

\subsection{$C^0$ bounds}

\begin{lemma}
  \label{lemma:barriers}
  Let $\bigl(\Mf^2,H\bigr)$ be
  a complete hyperbolic surface   and 
  let $\bigl(g(t)\bigr)_{t\in[0,T]}$ be a Ricci
  flow on $\m$ which is conformally equivalent to $H$. 
  \begin{compactenum}[(i)]
  \item If $g(t)$ is instantaneously complete, then 
    \begin{equation}
      \label{eq:lower-barrier}
      (2t) H \le g(t) \qquad\text{for all } t\in(0,T].
    \end{equation}
  \item 
    If there exists a constant $M>0$ such that $g(0)\le M H$, 
    then
    \begin{equation}
      \label{eq:upper-barrier}
      g(t) \le \bigl(2t+M\bigr) H \qquad\text{for all }
      t\in[0,T].
    \end{equation}
  \end{compactenum}
\end{lemma}
\begin{proof}
  (i) To establish the lower barrier \eqref{eq:lower-barrier} we use
  Chen's apriori estimate for the scalar curvature (Corollary
  \ref{cor:chen-lower-curv-bd-inst-complete}) to obtain the
  lower curvature bound $-\frac1{2t}\le K[g(t)]$ for all
  $t\in(0,T]$. Yau's Schwarz Lemma (Theorem \ref{thm:yau}) allows us then to
  compare $g(t)$ with  $H$, establishing \eqref{eq:lower-barrier}.  
  (For further details on Yau's result we refer to \cite[\S2]{GT10}.)

  (ii) Without loss of generality we may (possibly after lifting to its
  universal cover) assume $\Mf^2=\Disc$ and write 
  $g(t)=\ee^{2u(t)}|\dz|^2$.  
  To prove the upper barrier, consider for small $\delta>0$, 
  $u\bigr|_{\overline{\Disc_{1-\delta}}}$ and write the conformal factor of
  a complete Ricci flow on the disc of radius $1-\delta$ with Gaussian
  curvature initially $-M^{-1}$ as
  \[ h_\delta(t,z) := \log\frac{2(1-\delta)}{(1-\delta)^2-|z|^2} +
  \frac12\log\bigl(2t+M\bigr). \]
  Note that $u$ is continuous on $[0,T]\times\overline{\Disc_{1-\delta}}$
  and $h_\delta(t,z)\to\infty$ as $z\to\partial\Disc_{1-\delta}$ for all
  $t\in[0,T]$. 
  Also, with this choice of $M$, we have initially
  $u\bigr|_{\Disc_{1-\delta}}(0,\cdot) \le
  h_0\bigr|_{\Disc_{1-\delta}}(0,\cdot)\le h_\delta(0,\cdot)$. 
  Therefore the requirements of an elementary comparison principle for the
  Ricci flow (cf. Theorem \ref{thm:direct-comp-principle}) are fulfilled, and we may deduce that
  $u\bigr|_{\Disc_{1-\delta}}\le h_\delta$ holds throughout
  $[0,T]\times\Disc_{1-\delta}$. Since $h_\delta$ is continuous in $\delta$,
  letting $\delta\searrow0$ yields \eqref{eq:upper-barrier}. 
\end{proof}

\subsection{$C^k$ bounds}

In this section we bootstrap the estimates of the previous 
section to obtain estimates for higher derivatives.

\begin{lemma}
  \label{lemma:Ck-bounds}
  Let $\bigl(\ee^{2u(t)}|\dz|^2\bigr)_{t\in[0,T]}$ be an instantaneously complete
  Ricci flow on the unit disc $\Disc$ and for $r\in(0,1]$ let
  $H_r=\ee^{2h_r}|\dz|^2$ be the complete hyperbolic metric 
  on the disc $\Disc_r$ of radius $r$. If
  $\ee^{2u(0)|_{\Disc_r}}|\dz|^2\le MH_r$ for some $M>0$, then there
  exists for any $\delta\in\bigl(0,\min\{r,T\}\bigr)$ and 
  $k\in\mathbb N_0$ a constant
  $C_k=C_k(k,r,\delta,M)<\infty$ such that for all $t\in[\delta,T]$ 
  \begin{equation}
    \label{eq:Ck-bounds-cf}
    \Bigl\| u(t,\cdot)-\frac12\log(2t)
    \Bigr\|_{C^k(\Disc_{r-\delta},|\dz|^2)} \le C_k.
  \end{equation}
\end{lemma}
\begin{proof}
  Lemma \ref{lemma:barriers} provides the following upper and lower bounds for
  $u$,
  \begin{equation}
    \label{eq:existence-disc-barriers}
    \log\frac2{1-|z|^2} + \frac12\log(2t) \le u(t,z)
    \le \log\frac{2r}{r^2-|z|^2} + \frac12\log\bigl(2t+M\bigr)
  \end{equation}
  for all $t\in(0,T]$ and $z\in \Disc_r$.
  Since $\bigl|u\bigr|$ cannot be bounded uniformly away from $t=0$
  independently of $T$, consider the normalised flow $v(t)$
  defined by
  \[ v(t,z) := u(t,z) - s(t)\qquad\text{ where }\qquad
  s(t) := \frac12\log(2t), \]
  which evolves with the new time scale $s$ according to
  \begin{align}
    \label{eq:ev-mod-v}
    \pdds v &= \frac1{s'(t)} \frac{\partial v}{\partial t} 
    = \frac1{s'(t)}\left( \frac{\partial u}{\partial t} - s'(t)\right) 
    \nonumber\\
    &= \frac1{s'(t)}\Bigl(\ee^{-2u}\Delta u -  s'(t)\Bigr) 
    = \frac1{s'(t)}\Bigl(\ee^{-2s}\ee^{-2v}\Delta v -  s'(t)\Bigr) 
    \nonumber\\
    &= \ee^{-2v}\Delta v - 1 = \Div\bigl(\ee^{-2v}\mathrm{D} v\bigr) +
    2\ee^{-2v}\bigl|\mathrm{D}v\bigr|^2 -1.
  \end{align}
  From \eqref{eq:existence-disc-barriers} we get uniform bounds for
  $\bigl|v\bigr|$ away from $t=0$, which are independent of $T$, 
  \begin{equation}
    \label{eq:barriers-v}
    \log\frac2{1-|z|^2} \le v(t,z) \le \log\frac{2r}{r^2-|z|^2}
    + \frac12\log\frac{2t+M}{2t},
  \end{equation}
  for $z\in \Disc_r$.
  Indeed, fixing $\delta\in\bigl(0,\min\{r,T\}\bigr]$, 
  there exists a constant $C=C(M,\delta,r)>0$
  such that 
  \[ \sup_{[\nicefrac\delta2,T]\times\Disc_{r-\nicefrac\delta2}}\bigl|v\bigr|
  \le C<\infty. \]
  Thus the evolution equation \eqref{eq:ev-mod-v} for $v(s)$ is 
  uniformly parabolic on 
  $\bigl[s(\nicefrac\delta2),s(T)\bigr]\times\Disc_{r-\nicefrac\delta2}$, which
  allows us to apply standard 
  parabolic theory (e.g. \cite[Theorem V.1.1]{LSU68} to establish H\"older
  bounds on $v$ and then bootstrap with \cite[Theorem IV.10.1]{LSU68}) 
  to obtain,
  for any $k\in\mathbb N_0$,
   constants $C_k=C_k(k,\delta,C)>0$ such that  
  \[  \bigl\|\,v\,\bigr\|_{C^{k+\frac\alpha2,2k+\alpha}
    \bigl([\delta,T]\times(\Disc_{r-\delta},|\dz|^2)\bigr)} \le C_k \] 
  yielding \eqref{eq:Ck-bounds-cf}.
\end{proof}

\begin{thm}
  \label{thm:Ck-uniform-conv-disc}
  Let $(\m,H)$ be a complete hyperbolic surface, and 
  suppose $\bigl(g(t)\bigr)_{t\in[0,T]}$ is an instantaneously
  complete Ricci flow on $\Mf$
  which is conformally  equivalent  to $H$. If $g(0)\le MH$ for some constant
  $M>0$, then for all $k\in\mathbb N_0$ and any $\eta\in(0,1)$ and
  $\delta\in(0,T)$ (however small) there exists a 
  constant $C=C(k,\eta,\delta,M)<\infty$ such that for all $t\in[\delta,T]$
  there holds 
  \begin{equation}
    \label{eq:Ck-uniform-conv-disc}
    \left\|\, \frac1{2t}g(t) - H\, \right\|_{C^k(\Mf,H)} 
    \le \frac{C}{t^{1-\eta}}.
  \end{equation}
\end{thm}
\begin{proof}
  Fix any point $p\in\Mf$, and let $\pi:\Disc\to\Mf$ be a universal
  covering of $\Mf$ with $\pi(0)=p$. 
  Without loss of generality we may then write the pulled back metrics $\pi^*g(t)=\ee^{2u(t)}|\dz|^2$ 
  and $\pi^*H=\ee^{2h}|\dz|^2$, the latter being 
  the complete hyperbolic metric on the
  disc. We also have $\ee^{2u(0)} \le M\ee^{2h}$ by hypothesis.

  Using Lemma \ref{lemma:Ck-bounds} we obtain for every $k\in\mathbb N$
  constants $C'_k=C'_k(k,\delta,M)>0$ such that we have uniform $C^k$-bounds
  for all $t\in[\delta,T]$ 
  \begin{align}
    \label{eq:Ck-bounds-metric}
    \sup_{\Disc_{\nicefrac12}} \left| \mathrm D^k\left(\frac1{2t}\ee^{2(u(t)-h)} -
        1\right) \right| &=
    \sup_{\Disc_{\nicefrac12}} \left| \mathrm
      D^k\left(\ee^{2\left(u(t)-\frac12\log2t\right)}\ee^{-2h}\right) \right|
    \le C'_k. 
  \end{align}
  By virtue of Lemma \ref{lemma:barriers} there is a much stronger
  $C^0$-estimate for all $t\in(0,T]$ on $\Disc$ 
  \begin{equation}
    \label{eq:Ck-uniform-conv-D-C0}
    0 \le \frac1{2t}\ee^{2(u(t)-h)}-1 \le  \frac M{2t}.
  \end{equation}
  Now fix $\eta\in(0,1)$ and combine \eqref{eq:Ck-uniform-conv-D-C0} and
  \eqref{eq:Ck-bounds-metric} with the interpolation inequality 
  of Lemma \ref{lemma:interpolation-inequality} to obtain for every $k\in\mathbb N$ 
  constants $C''_k=C''_k(k,\eta,\delta,M)>0$ and $l=\left\lceil\nicefrac 
    k\eta\right\rceil$ such that for all $t\in[\delta,T]$ 
  \begin{align}
    \label{eq:Ck-uniform-conv-ii}
    \left| \mathrm D^k\left(\frac1{2t}\ee^{2(u(t)-h)} -1
      \right) \right|_{|\dz|^2}(0) 
  \leq
  C''_k  t^{-(1-\eta)}.
  \end{align}
  Note that the case $k=0$ of \eqref{eq:Ck-uniform-conv-ii} is already
  dealt with 
  by \eqref{eq:Ck-uniform-conv-D-C0}.
  Then we estimate using Lemma \ref{lemma:equiv-norms} (with constant
  $c=c(k)>0$) and
  \eqref{eq:Ck-uniform-conv-ii} for all $t\in[\delta,T]$
  \begin{align}
    \label{eq:Ck-uniform-conv-final-estim}
    \left| \nabla_H^k\left(\frac1{2t}g(t)-H\right)\right|_H(p) 
    &= \left| \nabla_{\pi^*H}^k\left(\frac1{2t}\ee^{2(u(t)-h)}-1\right)
      \pi^*H\right|_{\pi^*H}(0) \nonumber \\
    &\le c\sum_{j=0}^k \left| \mathrm
      D^j\left(\frac1{2t}\ee^{2(u(t)-h)}-1\right) 
    \right|_{|\dz|^2}(0) \nonumber \\
    &\le c\sum_{j=0}^k C''_j t^{-(1-\eta)} =: C'''_k
    t^{-(1-\eta)}.
  \end{align}
  Since $p\in\Mf$ was chosen arbitrarily and the constants $C'''_k$ are
  independent of $p$ (and of $\pi$), \eqref{eq:Ck-uniform-conv-final-estim}
  holds for all $p\in\Mf$ and we conclude that 
  with $C=C(k,\eta,\delta,M)>0$,
  \[ \left\|\, \frac1{2t}g(t) - H\, \right\|_{C^k(\Mf,H)} 
  \le 
  C\, t^{-(1-\eta)}  \]
  for all $t\in[\delta,T]$.
\end{proof}

\subsection{Curvature bounds and long time existence}

\begin{prop}
  \label{prop:uniform-curvature-bound}
  Let $\bigl(\Mf^2,H\bigr)$ be a complete hyperbolic surface
  and   
  $\bigl(g(t)\bigr)_{t\in[0,T]}$ an instantaneously complete 
  Ricci flow on $\m$, conformally equivalent to $H$.
  If $g(0)\le MH$ for some constant $M>0$, 
  then for all $\delta\in(0,T]$ there exists a constant
  $B=B(M,\delta)<\infty$ such that for all $t\in[\delta,T]$ there holds
  \[ \Bigl|K[g(t)] \Bigr| \le \frac Bt. \]
\end{prop}
\begin{proof}
  By Theorem \ref{thm:Ck-uniform-conv-disc} there exists a constant
  $C=C(\delta,M)>0$ such that for all $t\in[\delta,T]$
  \[ \Bigl| K[g(t)] \Bigr| = \frac1{2t} \left| K\Bigl[\frac1{2t}g(t)\Bigr]
  \right| \le \frac C{2t}. \]
\end{proof}
Note that from Theorem \ref{thm:Ck-uniform-conv-disc} we even have
\[ 2t\,K[g(t)] = K\left[\frac1{2t}g(t)\right] \; \longrightarrow\; -1
\quad\text{uniformly as }t\to\infty. \]

By the work of Shi and Chen-Zhu (Theorem \ref{shihamilton}) we can
state Hamilton's long time existence result \cite[Theorem 14.1]{Ham82} in the
more general setting of complete Ricci flows. Although 
we only need it on surfaces it is also true in higher dimensions.
\begin{lemma}
  \label{lemma:continue-bdd-solution}
  For some $T<\infty$ and $\kappa<\infty$ let $\bigl(g(t)\bigr)_{t\in[0,T]}$
  be a complete Ricci flow on a manifold $\Mf^n$ with bounded curvature 
  $\bigl|\Rm[g(t)]\bigr|\le\kappa$ for all $t\in[0,T]$. Then
  there exist constants 
  $\tau=\tau(\kappa,n)>0$, $\tilde\kappa = \tilde\kappa(\kappa,n)<\infty$
  and a smooth complete extension $\bigl(\tilde
  g(t)\bigr)_{t\in[0,T+\tau]}$ such that $\tilde g(t)=g(t)$ for all
  $t\in[0,T]$ and $\bigl|\Rm[\tilde g(t)]\bigr| \le \tilde\kappa$ for all
  $t\in[T,T+\tau]$. 

  Moreover, if for some $\varepsilon>0$ there exists another such
  complete extension $\bigl(\bar g(t)\bigr)_{t\in[0,T+\varepsilon]}$ with bounded
  curvature and $g(t)=\bar g(t)$ for all $t\in[0,T]$, then $\tilde g(t)=\bar
  g(t)$ for all $t\in\bigl[0,T+\min\{\tau,\varepsilon\}\bigr]$.
\end{lemma}
\begin{proof}
  By Theorem \ref{shihamilton} of Hamilton-Shi there exist a constant
  $\tau=\tau(\kappa,n)>0$ 
  and a complete Ricci flow $\bigl(\tilde g(t)\bigr)_{t\in[T,T+\tau]}$
  starting at $\tilde g(T)=g(T)$ with bounded curvature $\bigl|\Rm[\tilde
  g(t)]\bigr|\le\tilde\kappa<\infty$ for all $t\in[T,T+\tau]$. Combining both
  solutions we obtain the desired extension $\bigl(\tilde
  g(t)\bigr)_{t\in[0,T+\tau]}$ by setting $\tilde g(t)=g(t)$ for all $t\in[0,T)$. 

  If $\bigl(\bar g(t)\bigr)_{t\in[0,T+\varepsilon]}$ for some $\varepsilon>0$
  is another complete Ricci flow with bounded curvature extending $g(t)$, then $\bar
  g(t)=\tilde g(t)$ for all $t\in\bigl[0,T+\min\{\tau,\varepsilon\}]$ by
  Theorem \ref{shihamilton} of Hamilton or Chen-Zhu.
\end{proof}

\begin{cor}
  \label{cor:long-time-existence-disc}
  Let $\bigl(\Mf^2,H\bigr)$  be a complete hyperbolic surface, 
  and $\bigl(g(t)\bigr)_{t\in[0,T]}$ an instantaneously complete
  Ricci flow on $\Mf$, conformally equivalent to $H$.
  If $g(0)\le MH$ for some constant $M>0$, then there exists a unique
  instantaneously complete extension of $g(t)$
  defined for all time $t\in[0,\infty)$. 

  Moreover, if for some $\kappa_0<\infty$ there holds
  $K[g(t)]\le\kappa_0$ for all $t\in[0,T]$, then there exists a constant
  $\kappa=\kappa(M,T,\kappa_0)<\infty$ such that 
  \begin{equation}
    \label{eq:uniform-upper-curv-bd}
    K[g(t)] \le \kappa \qquad\text{for all } t\in[0,\infty).
  \end{equation}
\end{cor}
\begin{proof}
  The long-time existence is a direct consequence of the preceding Lemma
  \ref{lemma:continue-bdd-solution} and the apriori curvature bounds of
  Proposition \ref{prop:uniform-curvature-bound}: We have for all
  $t\in\bigl[\nicefrac T2,T\bigr]$
  \begin{equation}
    \label{eq:long-time-existence-curv-bd}
    \Bigl|K\bigl[g(t)\bigr]\Bigr| \le \frac Bt \le \frac{2B}T =: \kappa_1
    <\infty.     
  \end{equation}
  By Lemma \ref{lemma:continue-bdd-solution} there exist a $\tau>0$ depending
  only on $\kappa_1$ and a unique extension $\bigl(\tilde
  g(t)\bigr)_{t\in[0,T+\tau]}$ with the very same curvature bound $\kappa_1$ in
  \eqref{eq:long-time-existence-curv-bd} on the larger time interval
  $\bigl[\nicefrac T2,T+\tau\bigr]$. Iterating these arguments we obtain for
  any $j\in\mathbb N$ an extension $\bigl(\tilde g(t)\bigr)_{t\in[0,T+j\tau]}$
  with bounded curvature $\bigl|K[\tilde g(t)]\bigr|\le \kappa_1$ for all
  $t\in\bigl[\nicefrac T2,T+j\tau\bigr]$, and therefore we can continue $g(t)$
  to exist for all time $t\in[0,\infty)$. Since the curvature is bounded away
  from $t=0$, by Lemma \ref{lemma:continue-bdd-solution} the extension is also
  unique among other instantaneously complete extensions.
  
  In order to show the uniform upper bound for the curvature in
  \eqref{eq:uniform-upper-curv-bd}, note that  
  \eqref{eq:long-time-existence-curv-bd} is true for all
  $t\in\bigl[\nicefrac T2,\infty)$. Combining that bound with $\kappa_0$ for
  times $t\in[0,\nicefrac T2]$, we conclude the theorem with
  $\kappa:=\max\{\kappa_0,\kappa_1\}$. 
\end{proof}

\section{Existence}

\subsection{Existence of a maximally stretched solution on the disc $\Disc$}
\label{subsect:ex-disc}

In this section we prove the main existence and asymptotics
result in the case that the Ricci flow starts at a Riemannian surface
which is conformally the disc.
\begin{thm}
  \label{thm:exist-D}
  Let $(\Disc,H)$ be the complete hyperbolic disc and let
  $g_0$ be a smooth (possibly incomplete) Riemannian
  metric on $\Disc$ which is conformally equivalent to $H$.
  Then there exists a unique\footnote{In fact, if $\tilde g(t)$ is any maximally
    stretched Ricci flow for $t\in [0,T]$ with $\tilde g(0)=g_0$, then $\tilde
    g(t)=g(t)$ for all $t\in [0,T]$.}, maximally stretched and instantaneously
  complete Ricci flow $g(t)$ for all time $t\in[0,\infty)$
  with $g(0)=g_0$, and the
  rescaled flow converges smoothly locally
  \[ \frac1{2t} g(t)\;\longrightarrow\; H \quad\text{as}\quad
  t\to\infty. \] 
  Moreover, if $g_0\le M H$ for some (possibly large) $M>0$ then the convergence
  above is global: For any $k\in\mathbb N_0$ and $\eta\in(0,1)$ there is a
  constant $C=C(k,\eta,M)>0$ such that
  \[ \left\| \frac1{2t} g(t) - H \right\|_{C^k(\Disc,H)}\le \frac C{t^{1-\eta}} \]
for $t\geq 1$.
\end{thm}
\begin{proof}
  Without loss of generality, we can
  write $g_0=\ee^{2u_0}|\dz|^2$ for the initial metric and $H=\ee^{2h}|\dz|^2$
  for the complete hyperbolic metric on the disc $\Disc$.
  
  Let $\big(D_j\bigr)_{j\in\mathbb N}\subset\Disc$ be a suitable
  exhaustion of $\Disc$, 
  e.g. $D_j=\Disc_{1-\frac1{j+1}}$, and
  define $C_j:=\sup_{D_j} u_0<\infty$ and $\kappa^0_j:=\sup_{D_j}K[g_0]$.
  Furthermore let $H_j=\ee^{2h_j}|\dz|^2$ be the complete hyperbolic metric on
  the smaller disc $D_j$. Note that there holds  
  \begin{equation}
    \label{eq:locally-below-hyperbolic}
    g_0\bigr|_{D_j} \le \ee^{2C_j}H_j.   
  \end{equation}
  Then for each $j\in\mathbb N$, by virtue of Theorem
  \ref{ICexistence1}\footnote{Combining the techniques of this
  paper and of \cite{GT10} yields a simpler, more direct proof of
  this theorem.} 
  there exist constants $T_j=T_j(\kappa^0_j)>0$, 
  $\tilde\kappa_j=\tilde\kappa_j(T_j,\kappa^0_j)<\infty$, and
  a maximally stretched and instantaneously complete solution
  $\bigl(g_j(t)\bigr)_{t\in[0,T_j]}$ to the Ricci flow on $D_j$ with 
  $g_j(0)=g_0\bigr|_{D_j}$, and $K[g(t)]\le\tilde\kappa_j$ for all
  $t\in[0,T_j]$. 
  Since $g_j(0)\le \ee^{2C_j}H_j$,
  we may apply Corollary
  \ref{cor:long-time-existence-disc} to show that each $g_j(t)$ 
  can be extended to exist forever and has a uniform (in $t$)
  upper curvature bound
  $K[g_j(t)]\le\kappa_j=\kappa_j(C_j,\kappa^0_j)<\infty$.
  Define $u_j(t)$ such that  $\ee^{2u_j(t)}|\dz|^2=g_j(t)$. 
    
  Next observe that owing to Lemma \ref{lemma:compare-sequence},
  for all $(t,z)\in[0,\infty)\times\Disc$ and $j\in\mathbb
  N$ sufficiently large such that $z\in D_j$, the sequence
  $\bigl(u_j(t,z)\bigr)_{j\in\mathbb N}$ is (weakly) decreasing.
  Therefore, for any $\Om\subset\subset\Disc$ and $T\in (1,\infty)$,
  $u_j$ is uniformly bounded above (independently of $j$) 
  on $[0,T]\times\Om$ for sufficiently large $j$.
  On the other hand by Lemma \ref{lemma:barriers}(i) we know that the
  conformal factor of the `big-bang' Ricci flow $(2t)H$ is a
  lower barrier for each $g_j(t)$ i.e. 
  \begin{equation}
    \label{eq:exist-D-lower-barrier-k}
    h\bigr|_{D_j} +\frac12\log(2t) \le u_j(t)
  \end{equation}
  for all $t>0$. 
  Therefore for any $t_0\in (0,T)$, 
  $u_j$ is uniformly bounded below (independently of $j$)
  on $[t_0,T]\times\Om$.
  To obtain a uniform lower bound on $u_j$ near $t=0$, we will
  follow the ideas of \cite{Top10} and appeal to the 
  pseudolocality-type result 
  Theorem \ref{thm:chen-2d-local-curv-estim} of Chen.
Since $\Omega$ is compact and $g_0$ is smooth,
we can choose $r_0,v_0>0$ sufficiently small such that for all 
$p\in\Omega$ there holds, for sufficiently large $j$,
\begin{compactenum}[(i)]
\item $\gBall_{g_0}(p;r_0)\subset\subset D_{j}$, in particular 
    $\gBall_{g_j(t)}(p;r_0)\subset\subset D_j$ for all $t\in[0,T]$;
    \item $\bigl|K[g_0]\bigr| \le r_0^{-2}$ on $\gBall_{g_0}(p;r_0)$;
    \item $\Vol_{g_0}\gBall_{g_0}(p;r_0) \ge v_0 r_0^2$.
    \end{compactenum}
Therefore we may apply Theorem \ref{thm:chen-2d-local-curv-estim} 
to each such flow $g_j(t)$ and obtain a constant $\tau=\tau(v_0,r_0)\in(0,T]$ such that for sufficiently large $j$ and $t\in[0,\tau]$ 
    \begin{equation}
      \label{eq:nlic-curv-bound}
      \Bigl| K[g_j(t)] \Bigr| \le 2r_0^{-2}\quad\text{on }\Omega.
    \end{equation}
By inspection of the Ricci flow equation, this gives us 
a uniform lower bound on $u_j$ on $[0,\tau]\times\Om$ for
sufficiently large $j$.

Combining these estimates, we find that we have uniform
upper and lower bounds for the decreasing sequence 
$u_j$ on $[0,T]\times\Om$
(independent of $j$, for sufficiently large $j$) and thus we may apply 
parabolic regularity to get $C^k$ estimates on the functions
$u_j$ (uniform in $j$, for sufficiently large $j$) on any
compact subset of $[0,\infty)\times\Disc$.
  Therefore we may define a smooth function
  $u:[0,\infty)\times\Disc\to\mathbb R$ by 
  \[ u(t,z) := \lim_{j\to\infty} u_j(t,z), \]
  and the corresponding metric flow $g(t):= \ee^{2u(t)}|\dz|^2$
  must be a smooth Ricci flow, defined for all $t\in [0,\infty)$, 
  with $g(0)=g_0$.
By \eqref{eq:exist-D-lower-barrier-k} we also have $(2t)H\le g(t)$ on 
  $\Disc$ for all $t>0$, so $g(t)$ is instantaneously complete. 

  To see that $g(t)$ is maximally stretched, 
  let $\tilde u:[0,\varepsilon]\times\Disc\to\mathbb R$ be the conformal
  factor of any other Ricci flow with $\tilde u(0,\cdot)\leq u_0$, then the
  maximality of $u_j(t)$
  tells us that $\tilde u|_{D_j}(t,z) \leq u_j(t,z)$ for all $z\in D_j$ and
  $t\in [0,\varepsilon]$, and therefore (taking $j\to\infty$) $\tilde
  u(t,z)\leq u(t,z)$ for all $z\in\Disc$ and $t\in[0,\varepsilon]$. 
  Obviously $g(t)$ is unique amongst
  maximally stretched solutions (cf. Remark
  \ref{rmk:uniqueness-maximal-stretched}).

  We now conclude the proof by showing the asymptotic convergence:
  For fixed $\Omega\subset\subset\Disc$ define
  $\delta:=\dist(\Omega,\partial\Disc)>0$, and fix $k\in\mathbb N$.
  Using Lemma \ref{lemma:equiv-norms} (with constant $C'=C'(\delta,k)>0$) and
  Lemma \ref{lemma:Ck-bounds} we establish uniform $C^k$-bounds on $\Omega$
  for all $t\ge1$
  \begin{align}
    \label{eq:ex-D-local-Ck-bounds}
    \sup_\Omega \left|\nabla_H^k\left(\frac1{2t}g(t)-H\right)\right|_H &=
    \sup_\Omega \left|\nabla_H^k\left(\frac1{2t}g(t)\right)\right|_H 
    \le C' \left\| \frac1{2t}g(t)\right\|_{C^k(\Disc_{1-\delta},|\dz|^2)} \nonumber\\
    &= \sqrt2 C' \left\| \ee^{2\left(u(t)-\frac12\log2t\right)}
    \right\|_{C^k(\Disc_{1-\delta},|\dz|^2)} \nonumber\\ &\le
    C\Bigl(k,\delta,\sup_{\Disc_{1-\delta}}u_0\Bigr).
  \end{align}
  For any $r\in(0,1]$, let $H_r$ be the complete hyperbolic metric on the disc 
  $\Disc_r$ of radius $r$. Note that for $0<s\le r\le1$ we have
  $H\bigr|_{\Disc_s}\le H_r\bigr|_{\Disc_s}\le H_s$.
  Using Lemma \ref{lemma:barriers} we can estimate for any
  $r\in(1-\delta,1)$ with $M(r)= \inf\bigl\{M>0:
  g_0\bigr|_{\Disc_r}\le MH_r\bigr\}$ on $\Disc_r\supset\supset\Omega$
  \[ 0\le \frac1{2t}g(t)-H = \left(\frac1{2t}g(t) - H_r\right) + \bigl(H_r
  -H\bigr) \le \frac{M(r)}{2t}H_r + \bigl(H_r-H\bigr). \]
  Therefore 
  \begin{align}
    \label{eq:ex-D-local-C0-conv}
    \limsup_{t\to\infty} \left\| \frac1{2t} g(t)-H
    \right\|_{C^0(\Omega,H)} &\le \limsup_{t\to\infty} 
    \frac{M(r)}{2t} \Bigl\|H_r\Bigr\|_{C^0(\Omega,H)} +
    \Bigl\|H_r - H\Bigr\|_{C^0(\Omega,H)} \nonumber\\
    &= \Bigl\|H_r - H\Bigr\|_{C^0(\Omega,H)}
    \quad\stackrel{r\nearrow1}\longrightarrow\quad 0. 
  \end{align}
  Combining the local uniform $C^k$-bounds \eqref{eq:ex-D-local-Ck-bounds}
  with thc $C^0$-convergence \eqref{eq:ex-D-local-C0-conv} we obtain local
  convergence in the $C^k$ norm
  \[ \left\| \frac1{2t}g(t)-H\right\|_{C^k(\Omega,H)}
  \quad\stackrel{t\to\infty}\longrightarrow\quad 0. \]

  Finally, if $g_0 \le MH$ for some constant $M>0$ then the uniform
  and smooth convergence of $\frac1{2t}g(t)$ to $H$ as $t\to\infty$ is a
  consequence of Theorem \ref{thm:Ck-uniform-conv-disc}.
\end{proof}

\begin{proof}[Proof of \textbf{Theorem \ref{unfinished}}]
  By Corollary \ref{cor:chen-lower-curv-bd-inst-complete} and Corollary
  \ref{cor:long-time-existence-disc} $g_1(t)$ and $g_2(t)$ safisfy all
  conditions to compare each flow with the maximally stretched solution of
  Theorem \ref{thm:exist-D} using the same proof as \cite[Theorem 4.1]{GT10}.
\end{proof}

\subsection{Existence and uniqueness on the complex plane $\mathbb C$}
\label{subsect:ex-uniq-C}

Whilst Lemma \ref{lemma:barriers} gives a good lower barrier for
solutions on the disc $\Disc\subset\mathbb C$, it cannot be used directly
on the whole plane $\mathbb C$, because the plane does not admit a
hyperbolic metric. 
However, 
the following theorem provides such a uniform bound in that case by
considering the plane with a disc taken off which does admit a hyperbolic metric. 

\begin{thm}
  \label{thm:lower-barrier-C}
  Let $\bigl(\ee^{2u(t)}|\dz|^2\bigr)_{t\in[0,T]}$ be a smooth,
  instantaneously complete Ricci flow on the complex plane $\mathbb C$. Then 
  there exists a constant $C=C\bigl(u|_{[0,T]\times\Disc_2},T\bigr)<\infty$ such 
  that for all $|z|\ge2$ and $t\in(0,T]$ there holds
  \[ u(t,z) \ge -C -\log\bigl(|z|\log|z|\bigr) + \frac12\log(2t). \]
\end{thm}
\begin{proof}
  Pick any cutoff function $\varphi\in C^\infty_c\bigl(\Disc_2,[0,1]\bigr)$ with
  $\varphi\equiv1$ in $\Disc_{\nicefrac32}$, and define
  \[ \alpha := \sup_{[0,T]\times(\Disc_2\setminus\Disc_{\nicefrac32})} \bigl|u\bigr| +
  \bigl|\mathrm{D} u\bigr|_{|\dz|^2} + \bigl| \mathrm{D}^2u \bigr|_{|\dz|^2}. \]
  Furthermore, let $\ee^{2h}|\dz|^2$ be the complete hyperbolic metric on $\mathbb
  C\setminus\overline\Disc$, i.e. $h(z)=-\log\bigl(|z|\log|z|\bigr)$. Finally,
  consider an interpolated metric defined by the conformal factor 
  $v\in C^\infty\bigl([0,T]\times(\mathbb C\setminus\overline\Disc)\bigr)$ 
  given by
  \[ v(t,z) := \varphi(z)\cdot h(z) + \bigl(1-\varphi(z)\bigr)\cdot u(t,z). \]
  Note that by Corollary \ref{cor:chen-lower-curv-bd-inst-complete}, we have
  $K[u(t)]\ge-\frac1{2t}$ for all $t\in(0,T]$. Thus we can estimate the
  Gaussian curvature of $\ee^{2v(t)}|\dz|^2$ for all $t\in(0,T]$, by
  \[ K\bigl[v(t,z)\bigr] \ge\left.\begin{cases}
      -1 & \text{for } |z|\le\frac32\text{, i.e. where }v=h \\
      -\beta & \text{for } \frac32<|z|< 2\\
      -\frac1{2t} & \text{for } |z|\ge2\text{, i.e. where }v=u
    \end{cases}\right\} \ge\min\left\{ -\frac1{2t}, -\beta \right\} \]
  where $\beta=\beta(\alpha,\varphi)\ge1$. Comparing $\ee^{2v(t)}|\dz|^2$ with
  $\ee^{2h}|\dz|^2$ using Theorem \ref{thm:yau} yields
  for all $(t,z)\in(0,T]\times(\mathbb C\setminus\Disc_2)$
  \[ u(t,z) = v(t,z) \ge h(z) + \frac12\log\Bigl(\min\bigl\{ 2t, \beta^{-1}
  \bigr\}\Bigr) \ge -C -\log\bigl(|z|\log|z|\bigr) +\frac12\log(2t), \]
  defining $C:=\max\Bigl\{\frac12\log(2T),0\Bigr\}+\frac12\log\beta<\infty$.
\end{proof}
\begin{cor}
  \label{cor:comparison-principle-C}
  Let $\bigl(g_1(t)\bigr)_{t\in[0,T]}$ and $\bigl(g_2(t)\bigr)_{t\in[0,T]}$ be   two Ricci flows on $\mathbb C$ conformally equivalent
  to $|\dz|^2$. If $g_2(t)$ is instantaneously complete and $g_1(0)\le 
  g_2(0)$, then $g_1(t)\le g_2(t)$ for all $t\in[0,T]$. 
\end{cor}
\begin{proof}
  By Theorem \ref{thm:lower-barrier-C} the conformal factor $u(t)$ of
  $g_2(t)=\ee^{2u(t)}|\dz|^2$ satisfies the decay condition
  \eqref{eq:comp-C-decay-cond} of the comparison principle Theorem
  \ref{thm:comparison-C} by Rodriguez, Vazquez and Esteban, and the
  Corollary's statement follows.
\end{proof}

It is now easy to prove Theorem \ref{not_hyp_unique} by dividing 
into the two cases that the universal cover of $(\m,g_0)$
is conformally $\mathcal S^2$ or $\C$.

\begin{thm}
  \label{thm:exist-unique-C}
  Let $g_0$ be a smooth (possibly incomplete) Riemannian metric on the complex  plane $\mathbb C$, which is conformally equivalent to 
 $|\dz|^2$.   
  Then there exists a unique instantaneously complete Ricci flow 
  $g(t)$ with  $g(0)=g_0$ for all $t\in[0,T)$ up to a maximal time
  \[ T = \frac1{4\pi}\Vol_{g_0} \mathbb C \le\infty. \]
  Uniqueness here is in the sense that any other 
  instantaneously complete Ricci flow on $\C$ with initial
  metric $g_0$ will agree with $g(t)$ while they both exist.
  Moreover, the flow $g(t)$ is maximally stretched.
\end{thm}
\begin{proof}
  In \cite{DD96} DiBenedetto and Diller showed that
  if $\Vol_{g_0}\mathbb C<\infty$, there exists a maximally stretched and
  instantaneously complete solution $g(t)$ to the Ricci flow for $t\in[0,T)$
  up to a maximal time $T=\frac1{4\pi}\Vol_{g_0} \mathbb C$, with
  volume decaying linearly to zero as $t$ increases to $T$.  

  In the case of infinite volume $\Vol_{g_0}\mathbb C=\infty$ we are going
  to approximate the solution by a sequence of finite volume solutions: 
  Define a weakly increasing sequence $\bigl(g_{0,j}\bigr)_{j\in\mathbb N}$
  that converges smoothly locally to $g_0$, but has finite volume
  $\Vol_{g_{0,j}}\mathbb C<\infty$ for all $j\in\mathbb N$. 
  Then apply for each $j\in\mathbb N$ DiBenedetto and Diller's existence
  theorem to obtain instantaneously complete solutions $g_j(t)$ 
  with $g_j(0)=g_{0,j}$, defined for all 
  $t\in[0,T_j)$ up to a maximal time $T_j=\frac1{4\pi}\Vol_{g_{0,j}} \mathbb
  C\to\infty$ as $j\to\infty$. These instantaneously complete solutions allow
  us to use Corollary \ref{cor:comparison-principle-C} to see that the
  sequence $\bigl(g_j(t)\bigr)_{j\in\mathbb N}$ is also weakly increasing for
  all $j\ge j_0$ such that $t\le T_{j_0}$. By Lemma \ref{lemma:barriers}(ii)
  there is also a uniform upper barrier on compact subsets of
  $[0,\infty)\times\mathbb C$, hence by parabolic regularity
  theory, the sequence
  $\bigl(g_j(t)\bigr)_{j\in\mathbb N}$ converges locally smoothly on
  $[0,\infty)\times\mathbb C$ to a smooth Ricci flow $g(t)$
  with $g(0)=g_0$.
  Since for all $t\in(0,\infty)$ we have $g(t)\ge g_j(t)$ for all $j$
  sufficiently large so that $T_j>t$, we find that
  $g(t)$ is also complete. 

The claimed uniqueness and the maximally stretched property 
are a direct consequence of Corollary
\ref{cor:comparison-principle-C}.

An alternative way of constructing $g(t)$ would be to follow the strategy of Theorem \ref{thm:exist-D} but choosing $D_j$ to
be an exhaustion of $\C$ rather than $\Disc$.
This gives an instantaneously complete solution which by uniqueness
must agree with the solution above. 
\end{proof}

\subsection{Proof of Theorems \ref{thm:2d-existence} and \ref{thm:no-weird-solution}}
\label{last_sect}

\begin{proof}[Proof of \textbf{Theorem \ref{thm:2d-existence}}]
  Let $\pi:\tilde\Mf\to\Mf$ be the universal covering of $\Mf$ with the
  lifted metric $\tilde g_0=\pi^*g_0$, and let $\Gamma<\Isom(\tilde\Mf,\tilde
  g_0)$ be the discrete subgroup of the isometry group, isomorphic to
  $\pi_1(\Mf)$, such that $\Mf\cong\tilde\Mf/\Gamma$. 
  Then as is implicit throughout the paper, by virtue of the uniformisation theorem, $\bigl(\tilde\Mf,\tilde
  g_0\bigr)$ is conformally equivalent  
  either to the sphere $\mathcal S^2$, to the complex plane
  $\mathbb C$ or to the open unit disc $\Disc$.

  Therefore by Theorems \ref{thm:exist-unique-closed-surface}, \ref{thm:exist-unique-C} and
  \ref{thm:exist-D} resp.~there exists an instantaneously 
  complete and maximally stretched Ricci flow 
  $\bigl(\tilde g(t)\bigr)_{t\in[0,T)}$ on $\tilde\Mf$ with
  $\tilde g(0)=\tilde g_0$ up to a maximal time
  \[ T = \begin{cases}
    \frac1{8\pi} \Vol_{\tilde g_0}\tilde\Mf & 
    \text{if }\tilde\Mf\cong\mathcal S^2, \\
    \frac1{4\pi} \Vol_{\tilde g_0}\tilde\Mf & 
    \text{if }\tilde\Mf\cong\mathbb C, \\
    \qquad\infty & \text{if }\tilde\Mf\cong\Disc.
  \end{cases} \]
  By Lemma \ref{lemma:preserve-isom-RF}, $\Gamma$ acts by isometries on
  $\bigl(\tilde\Mf,\tilde g(t)\bigr)$ for every $t\in[0,T)$, so 
  we may quotient $\tilde g(t)$ to obtain uniquely a maximally stretched and
  instantaneously complete solution $g(t)=\pi_*\tilde g(t)$ on
  $\Mf=\tilde\Mf/\Gamma$ for all $t\in[0,T)$ with $g(0)=g_0$. 

  Finally, using the relation $|\Gamma|\cdot\Vol_{g_0}\Mf=\Vol_{\tilde
    g_0}\tilde\Mf$ we are going to phrase the maximal time $T$ in terms of
  $\bigl(\Mf,g_0\bigr)$ by distinguishing the only cases: 
  \[ T = \begin{cases}
    \frac1{8\pi}\Vol_{g_0}\Mf<\infty & \text{if $\tilde\Mf\cong\mathcal S^2$ and 
      $|\Gamma|=1$}\quad\Longrightarrow\quad\Mf\cong\mathcal S^2\\  
    \frac1{4\pi}\Vol_{g_0}\Mf<\infty & \text{if $\tilde\Mf\cong\mathcal S^2$ and
      $|\Gamma|=2$}\quad\Longrightarrow\quad\Mf\cong\mathbb R\!P^2\\
    \frac1{4\pi}\Vol_{g_0}\Mf\le\infty & \text{if $\tilde\Mf\cong\mathbb C$\hspace{.95ex} and
      $|\Gamma|=1$}\quad\Longrightarrow\quad\Mf\cong\mathbb C\\
    \qquad\infty & \text{if $\tilde\Mf\cong\mathbb C$\hspace{.95ex} and
      $|\Gamma|=\infty$}\\
    \qquad\infty & \text{if $\tilde\Mf\cong\Disc$}.
  \end{cases} \]
  The local convergence in the hyperbolic case follows also 
  from Theorem \ref{thm:exist-D}: 
  For convergence on an arbitrary ball $\gBall_H(p;r)$ in
  $\bigl(\Mf,H\bigr)$ choose a point $\tilde p\in \pi^{-1}(p)$ and 
  consider the ball $\gBall_{\pi^*H}(\tilde p;r)\subset\Disc$.
  Then the local smooth convergence of $\frac1{2t}g(t)$ to $H$ on
  $\gBall_H(p;r)$ is a consequence of Theorem \ref{thm:exist-D} which we can
  apply to show smooth convergence of $\frac1{2t}\tilde g(t)$ to $\pi^*H$ on
  $\gBall_{\pi^*H}(\tilde p;r)$. 
  The global convergence in the case that $g_0\le  
  MH$ for some $M>0$ is the statement of Theorem
  \ref{thm:Ck-uniform-conv-disc}.
\end{proof}

Last we provide the proof that in the case of a complete initial surface
with bounded curvature our solution does not differ from the Hamilton-Shi Ricci
flow. 

\begin{proof}[Proof of \textbf{Theorem \ref{thm:no-weird-solution}}]
  Since $g(t)$ is the unique maximally stretched solution, 
  it agrees with the solution from Theorem \ref{ICexistence1}
  which has a uniform upper bound to the 
  curvature. On the other hand, Chen's apriori estimate Theorem
  \ref{thm:chen-lower-curv-bd} provides also a uniform lower bound to the
  curvature. From \cite{Shi89} or \cite{Ham82} we know that $\tilde g(t)$ is
  complete and of bounded curvature. Therefore by Theorem 
  \ref{thm:uniqueness-complete-curv-bd} $g(t)$ and $\tilde g(t)$ must coincide.
\end{proof}

\begin{appendix}
  \section{Comparison principles}
  In this appendix we state different comparison principles and some direct
  consequences. We start with an elementary one whose proof can be found in
  \cite{GT10}. 
  \begin{thm}{\rm\cite[Theorem A.1]{GT10}}
    \label{thm:direct-comp-principle}
    Let $\Omega\subset\mathbb C$ be an open, bounded domain and for some
    $T>0$ let $u\in C^{1,2}\bigl((0,T)\times\Omega\bigr)\cap
    C\bigl([0,T]\times\bar\Omega\bigr)$ and $v\in
    C^{1,2}\bigl((0,T)\times\Omega\bigr)\cap
    C\bigl([0,T]\times\Omega\bigr)$ both be solutions of the Ricci flow
    equation \eqref{eq:ricci-flow-cf} for the conformal factor of the metric.
    Furthermore, suppose that for each $t\in[0,T]$ we have
    $v(t,z)\to\infty$ as $z\to\partial\Omega$.
    If $v(0,z)\ge u(0,z)$ for all $z\in\Omega$, then $v\ge u$ on
    $[0,T]\times\Omega$.
  \end{thm}
  As a direct consequence we can show that Ricci flows on discs of different
  sizes stay ordered.
  \begin{lemma}
    \label{lemma:compare-sequence}
    For $0<r<R<\infty$ let $\bigl(g_1(t)\bigr)_{t\in[0,T]}$ on $\Disc_R$
    and $\bigl(g_2(t)\bigr)_{t\in[0,T]}$ on $\Disc_r$ be two solutions to the
    Ricci flow which are conformally equivalent on $\Disc_r$ and satisfy
    \begin{compactenum}[(i)]
    \item $g_1(0)\bigr|_{\Disc_r} \le g_2(0)$,
    \item $g_2(t)$ is complete for all $t\in(0,T]$,
    \item there exists a constant $\kappa\in(0,\infty)$ such that  
      $K[g_2(t)]\le\kappa$ for all $t\in[0,T]$.
    \end{compactenum}
    Then $g_1(t)\bigr|_{\Disc_r}\le g_2(t)$ for all $t\in[0,T)$.
  \end{lemma}
  \begin{proof}
    Without loss of generality write $g_1(t)=\ee^{2u(t)}|\dz|^2$ and
    $g_2(t)=\ee^{2v(t)}|\dz|^2$ with $u(t)\in C^\infty(\Disc_R)$ and $v(t)\in
    C^\infty(\Disc_r)$. 
    For any $\delta\in(0,T)$ define
    \[ v_\delta(t,z) := v\bigl(\ee^{-2\kappa\delta}(t+\delta),z\bigr)+\kappa\delta
    \qquad \text{for } (t,z)\in[0,T-\delta]\times\Disc_r, \]
    which is a slight adjustment of $v$, again a solution to the Ricci flow
    \eqref{eq:ricci-flow-cf}
    \[ \left(\pddt v_\delta-\ee^{-2v_\delta}\Delta v_\delta\right)(t,z)
    = \ee^{-2\kappa\delta} \left(\pddt v-\ee^{-2v}\Delta v\right)
    \bigl(\ee^{-2\kappa\delta}(t+\delta),z\bigr) = 0. \]
    
    In order to compare $u\bigr|_{\overline{\Disc_r}}$ and $v_\delta$, 
    we are going to check the requirements of Theorem
    \ref{thm:direct-comp-principle}: For the conformal
    factor of the restricted metric we have $u\bigr|_{\overline{\Disc_r}}\in
    C\bigl([0,T]\times\overline{\Disc_r}\bigr)$. From (ii) and Lemma
    \ref{lemma:barriers}(i), we have that for all $t\in[0,T-\delta]$   
    \[ v_\delta(t,z)\ge\log\frac{2r}{r^2-|z|^2}+\frac12\log 2(t+\delta)\to\infty 
    \quad\text{as}\quad z\to\partial\Disc_r. \]
    Finally to check the initial condition use 
    the uniform upper bound $\kappa$ for the curvature of $v$ from (iii) to
    integrate \eqref{eq:ricci-flow-cf}, and we may estimate
    \[ v_\delta(0,z) =  v\bigl(\ee^{-2\kappa\delta}\delta,z\bigr)+\kappa\delta
    \ge v(0,z) - \kappa\ee^{-2\kappa\delta}\delta + \kappa\delta
    \stackrel{(i)}{\ge} u(0,z) \]  
    for all $z\in\Disc_r$. Thus by Theorem \ref{thm:direct-comp-principle} there holds
    $v_\delta(t,z)\ge u(t,z)$ for all $(t,z)\in[0,T-\delta]\times\Disc_r$ and
    all $\delta\in(0,T)$. Given any $(t,z)\in[0,T)\times\Disc_r$, we may
    conclude  
    \[ u(t,z) \le \lim_{\delta\searrow0} v_\delta(t,z) = v(t,z). \]
  \end{proof}
  The following more geometrical comparison principle from \cite{GT10}
  will us allow to give a simple proof of the uniqueness of complete
  Ricci flows with bounded curvature on surfaces.
  \begin{thm}{\rm\cite[Theorem 4.2]{GT10}}
    \label{thm:geom-comp-principle}
    For some $T>0$ let $\bigl(g_1(t)\bigr)_{t\in[0,T]}$ and
    $\bigl(g_2(t)\bigr)_{t\in[0,T]}$ be two conformally equivalent Ricci flows
    on a surface $\Mf^2$, and define $Q:[0,T]\times\m\to\R$ to be 
    the function for which $g_1(t)=\ee^{2Q(t)}g_2(t)$.
    Suppose further that $g_2(t)$ is complete for each $t\in [0,T]$
    and that for some constant $C\ge0$ we have
    \[ \text{(i) } \bigl|K[g_2]\bigr|\le C,\qquad 
    \text{(ii) } K[g_1]\le C, \qquad
    \text{(iii) } Q\le C \]
    on $[0,T]\times\m$. If $g_1(0)\le g_2(0)$, then 
    $g_1(t)\le g_2(t)$ for all $t\in [0,T]$.
  \end{thm}

  The contribution of Chen-Zhu \cite{CZ06} to Theorem \ref{shihamilton}
  was the uniqueness in the complete case. In our situation where
  the underlying manifold is two-dimensional, this is far simpler
  to prove than in the general case; the statement and proof
  are as follows.
  \begin{thm}
    \label{thm:uniqueness-complete-curv-bd}
    Let $\bigl(g_1(t)\bigr)_{t\in[0,T]}$ and $\bigl(g_2(t)\bigr)_{t\in[0,T]}$
    be two complete Ricci flows on a surface $\Mf^2$, with uniformly bounded
    curvature.   
    If $g_1(0)=g_2(0)$, then $g_1(t)=g_2(t)$ for all $t\in[0,T]$. 
  \end{thm}
  \begin{proof}
    With respect to a local complex coordinate $z$, let us write
    $g_1(t)=\ee^{2u(t)}|\dz|^2$ and $g_2(t)=\ee^{2v(t)}|\dz|^2$ for some
    locally defined functions $u(t)$ and $v(t)$, and define $Q:=u-v$
    globally.
    Observe that $Q$ is uniformly bounded since $Q=0$ at $t=0$
    and the curvature
    of $g_1(t)$ and $g_2(t)$ is uniformly bounded
    \[ \left|\pddt Q\right| = \Big| K[u_2] - K[u_1] \Bigr| \le C<\infty 
    \qquad\Longrightarrow\qquad \Bigl|Q\Bigr| \le CT. \]
    Therefore we may apply Theorem \ref{thm:geom-comp-principle} twice to obtain
    $g_1(t)\le g_2(t)$ and $g_2(t)\le g_1(t)$ for all $t\in[0,T]$.
  \end{proof}

  We will also require an obvious uniqueness property of maximally stretched solutions: 
  \begin{rmk}
    \label{rmk:uniqueness-maximal-stretched}
    Let $\big(g(t)\bigr)_{t\in[0,T]}$  and $\big(\tilde
    g(t)\bigr)_{t\in[0,\tilde T]}$ be two conformally equivalent and maximally
    stretched Ricci flows on $\Mf^2$ with $g(0)=\tilde g(0)$. 
    Then $g(t)=\tilde g(t)$ for all $t\in\bigl[0,\min\{T,\tilde T\}\bigr]$.
  \end{rmk}
One application of this uniqueness property is the preservation 
of the isometry group under a maximally stretched Ricci flow:
  \begin{lemma}
    \label{lemma:preserve-isom-RF}
    Let $\big(g(t)\bigr)_{t\in[0,T]}$ be a maximally stretched
    Ricci flow on $\Mf^n$. Then the isometry group does not shrink under the
    flow: $\Isom\bigl(\Mf,g(0)\bigr)\subset \Isom\bigl(\Mf,g(t)\bigr)$ for all
    $t\in[0,T]$. 
  \end{lemma}
  \begin{proof}
    Pick any $\phi\in\Isom\bigl(\Mf,g(0)\bigr)$. Since the Ricci flow is
    invariant under diffeomorphisms, $\phi^*g(t)$ is again a solution to the
    Ricci flow with $\phi^* g(0)=g(0)$ for all $t\in[0,T]$. 
    Because $g(t)$ is maximally stretched, we have $\phi^* g(t)\le g(t)$ and
    also $(\phi^{-1})^*g(t)\le g(t)$ for all $t\in[0,T]$. 
    Pulling back the latter inequality by $\phi$ yields $g(t)\le\phi^*g(t)$. 
    Therefore $g(t)=\phi^*g(t)$ and $\phi\in\Isom\bigl(\Mf,g(t)\bigr)$ for all
    $t\in[0,T]$. 
  \end{proof} 

  The last comparison principle is a result from the theory of the logarithmic
  fast diffusion equation on $\mathbb C$.
  
  \begin{thm}{\rm(Variant of Rodriguez-Vazquez-Esteban \cite[Corollary 2.3]{RVE97}.)}
    \label{thm:comparison-C}
    Let $\bigl(\ee^{2u(t)}|\dz|^2\bigr)_{t\in[0,T]}$ and 
    $\bigl(\ee^{2v(t)}|\dz|^2\bigr)_{t\in[0,T]}$ be two
    Ricci flows on the plane $\mathbb C$. 
If there exists $C<\infty$ such that 
$u(t)$ satisfies the decay condition 
    \begin{equation}
      \label{eq:comp-C-decay-cond}
      u(t,z) \ge -C -\log\bigl(|z|\log|z|\bigr) + \frac12\log(2t)
      \quad\text{for all }(t,z)\in(0,T)\times\mathbb C
    \end{equation}
    and $u(0)\ge v(0)$, then $u(t)\ge v(t)$ for all
    $t\in[0,T]$.  
  \end{thm}

  \section{Further supporting results}
  In \cite{Che09} Chen proves a very general apriori estimate for the scalar
  curvature of a Ricci flow without requiring anything but the completeness of
  the solution. 
  \begin{thm}{\rm(Chen \cite[Corollary 2.3(i)]{Che09}.)}
    \label{thm:chen-lower-curv-bd}
    Let $\bigl(g(t)\bigr)_{t\in[0,T]}$ be a smooth complete Ricci flow on a
    manifold $\Mf^n$. 
    If $R[g(0)]\ge -\kappa$ for some $\kappa\in[0,\infty]$, then 
    \[ R\bigl[g(t)\bigr] \ge -\frac n{2t+\frac n\kappa}\qquad\text{for all
    }t\in[0,T]. \] 
  \end{thm}
  A direct transfer to our situation where we have an instantaneously complete
  solution on a surface is:
  \begin{cor}
    \label{cor:chen-lower-curv-bd-inst-complete}
    Let $\bigl(g(t)\bigr)_{t\in[0,T]}$ be a smooth instantaneously 
    complete Ricci flow on a surface $\Mf^2$. Then  
    \[ K\bigl[g(t)\bigr] \ge -\frac1{2t}\qquad\text{for all }t\in(0,T]. \]
  \end{cor}
  
  For further applications and a simple proof of the following special case of
  the Schwarz lemma of Yau, see \cite{GT10}.
  \begin{thm}{\rm (Yau \cite{Yau73}.)}
    \label{thm:yau}
    Let $\left(\Mf_1,g_1\right)$ and $\left(\Mf_2,g_2\right)$ be two Riemannian
    surfaces without boundary. If
    \begin{compactenum}[(i)]
    \item $\left(\Mf_1,g_1\right)$ is complete,
    \item $K[g_1]\ge -a_1$ for some number $a_1\ge0$, and
    \item $K[g_2]\le -a_2<0$,
    \end{compactenum}
    then any conformal map $f:\Mf_1\to\Mf_2$ satisfies
    \[ f^*(g_2)\le\frac{a_1}{a_2} g_1. \]
  \end{thm}
  
A more elaborate argument of Chen leads to the following
pseudolocality-type result giving 2-sided estimates on the curvature.
  \begin{thm}{\rm (Chen \cite[Proposition 3.9]{Che09}.)}
    \label{thm:chen-2d-local-curv-estim}
    Let $\bigl(g(t)\bigr)_{t\in[0,T]}$ be a smooth Ricci flow
    on a surface $\Mf^2$. If we have for some $p\in\Mf$, $r_0>0$ and $v_0>0$
    \begin{compactenum}[(i)]
    \item $\gBall_{g(t)}(p;r_0)\subset\subset\Mf$ for all $t\in[0,T]$;
    \item $\Bigl| K[g(0)]\Bigr| \le r_0^{-2}$ on $\gBall_{g(0)}(p;r_0)$;
    \item $\Vol_{g(0)}\gBall_{g(0)}(p;r_0) \ge v_0 r_0^2$,
    \end{compactenum}
    then there exists a constant $C=C(v_0)>0$ such that for all
    $t\in\bigl[0,\min\bigl\{T,\frac1C r_0^2\bigr\}\bigr]$ 
    \[ \Bigl| K[g(t)] \Bigr| \le 2r_0^{-2} 
    \qquad\text{on } \gBall_{g(t)}\Bigl(p;\frac{r_0}2\Bigr). \]    
  \end{thm}

Occasionally we will have to switch between equivalent metrics
in arguments, and will use the following elementary fact:
  \begin{lemma}
    \label{lemma:equiv-norms}
    Let $\bigl(\Disc,H\bigr)$ be the complete hyperbolic disc and
    $T\in\Gamma\bigl(\TD[(r,s)]\bigr)$ any $(r,s)$ tensor field. Then
    for every $k\in\mathbb N_0$ and $\varrho\in(0,1)$ there exists a constant
    $C=C(k,\varrho,r,s)>0$ 
    such that 
    \[ \frac1C \bigl\|T\bigr\|_{C^k(\Disc_\varrho,|\dz|^2)} \le 
    \bigl\|T\bigr\|_{C^k(\Disc_\varrho,H)} \le
    C \bigl\|T\bigr\|_{C^k(\Disc_\varrho,|\dz|^2)}. \]
    In particular, we have
    \[ \frac1C \sum_{j=0}^k \Bigl|\nabla^j_{|\dz|^2} T\Bigr|_{|\dz|^2} (0)
    \le \sum_{j=0}^k \Bigl|\nabla^j_H T\Bigr|_H (0)
    \le C \sum_{j=0}^k \Bigl|\nabla^j_{|\dz|^2} T\Bigr|_{|\dz|^2} (0). \]
  \end{lemma}
In order to bootstrap $C^0$ convergence into $C^k$ convergence
(using $C^l$ bounds) we will need to be able to interpolate:
  \begin{lemma}
    \label{lemma:interpolation-inequality}
    Let $u:\Ball^n\to[-1,1]$ be a smooth function such that
    for all $k\in\mathbb N$ 
    \[ \bigl\| \mathrm D^ku \bigr\|_{L^\infty(\Ball)} < \infty .\]
    Then for all $k\in\mathbb N$ and $\eta\in(0,1)$ there exist constants
    $C=C(k,\eta)>0$ and $l:=\left\lceil\nicefrac k\eta\right\rceil$ such that
    \begin{equation}
      \label{eq:interpolation-inequality}
      \bigl| \mathrm D^k u \bigr|(0) \le
      C \left( 1+ \bigl\| \mathrm D^lu
        \bigr\|_{L^\infty(\Ball)} \right) 
      \bigl\| u \bigr\|^{1-\eta}_{L^\infty(\Ball)}. 
    \end{equation}
  \end{lemma}
  \begin{proof}
    By \cite[Theorem 7.28]{GT01} (for example)
    for arbitrary $0<k<l$ there is a constant
    $C=C(l)>0$ such that for any smooth function $v\in C^\infty(\Ball)\cap
    L^\infty(\Ball)$ with bounded derivatives
    \begin{equation}
      \label{eq:ii-1}
      \bigl\| \mathrm D^k v \bigr\|_{L^\infty(\Ball)} \le C\left(
        \bigl\| v \bigr\|_{L^\infty(\Ball)} + \bigl\| \mathrm D^l v
        \bigr\|_{L^\infty(\Ball)} \right).  
    \end{equation}
    Now define $v(x):= u(\varepsilon x)$ for all $x\in\Ball$ with
    $\varepsilon=\bigl\| u \bigr\|_{L^\infty(\Ball)}^{\frac1l}\in(0,1]$ and
    choosing $l=\left\lceil\nicefrac k\eta\right\rceil$ we estimate with
    \eqref{eq:ii-1}  
    \begin{align*}
      \label{eq:ii-2}
      \bigl\| \mathrm D^k u \bigr\|_{L^\infty(\Ball_\varepsilon)} 
      = \varepsilon^{-k} \bigl\| \mathrm D^k v \bigr\|_{L^\infty(\Ball)} 
      &\le C\varepsilon^{-k}\left(\bigl\| v \bigr\|_{L^\infty(\Ball)}
        + \bigl\|\mathrm D^l v \bigr\|_{L^\infty(\Ball)}\right) \nonumber\\ 
      &\le C \left(\varepsilon^{-k} \bigl\| u
        \bigr\|_{L^\infty(\Ball_\varepsilon)} +
        \varepsilon^{l-k} \bigl\|\mathrm D^l u
        \bigr\|_{L^\infty(\Ball_\varepsilon)}\right)\nonumber\\
      &\le C \left( 1+ \bigl\|\mathrm D^l u
        \bigr\|_{L^\infty(\Ball)} 
      \right)\bigl\| u \bigr\|^{1-\frac kl}_{L^\infty(\Ball)} \nonumber\\
      &\le C \left( 1+ \bigl\|\mathrm D^l u
        \bigr\|_{L^\infty(\Ball)} 
      \right)\bigl\| u \bigr\|^{1-\eta}_{L^\infty(\Ball)}.
    \end{align*}
  \end{proof}
  
\end{appendix}

\bibliography{2d_ricci_flow}
\end{document}